\documentclass[english,11pt]{amsart}
\usepackage{amsopn,amsthm,amsfonts,amsmath,amssymb,amscd}
\usepackage{babel}
\usepackage{amssymb,tikz}

\newtheorem{Theorem}{Theorem}[section]
\newtheorem{Lemma}[Theorem]{Lemma}

\newtheorem{Example}[Theorem]{Example}
\newtheorem{Remark}[Theorem]{Remark}
\newtheorem{Def}[Theorem]{Definition}

\newenvironment{Proof*}{{\it Proof.}}

\newcommand{\ZZ}{\mathbb{Z}}

\newcommand{\BB}{\mathcal{B}}

\newcommand{\DEF}[1]{\emph{#1}}

\newcommand{\pat}{{\rm patt}}
\newcommand{\MD}{{\rm dim_M}}

\newcommand{\diam}[1]{{\rm diam}(#1)}

\begin{document}

\title{The metric dimension of the zero-divisor graph of a matrix semiring}
\author{David Dol\v zan}
\date{\today}

\address{D.~Dol\v zan:~Department of Mathematics, Faculty of Mathematics
and Physics, University of Ljubljana, Jadranska 21, SI-1000 Ljubljana, Slovenia; e-mail: 
david.dolzan@fmf.uni-lj.si}

 \subjclass[2010]{}
 \keywords{semiring, zero-divisior, graph, metric dimension}
  \thanks{The author acknowledges the financial support from the Slovenian Research Agency  (research core funding No. P1-0222)}

\bigskip

\begin{abstract} 
We determine the metric dimension of the zero-divisor graph of the matrix semiring over a commutative entire antinegative semiring.
\end{abstract}

\maketitle 

\section{Introduction}

\bigskip

One of the most important and active areas in algebraic combinatorics is the study
of graphs associated with different algebraic structures. This field has attracted many researchers during the past 30 years. 

One of the most basic concepts in the study of (semi)rings is the notion of a zero-divisor. Thus, in 1988, Beck \cite{beck88} first introduced the concept of the zero-divisor graph of a commutative ring. In 1999, Anderson and Livingston \cite{andliv99} made a slightly different definition of the zero-divisor graph in order to be able to investigate the zero-divisor structure of commutative rings. 
In 2002, Redmond \cite{redmond02} extended this definition to also include  non-commutative rings. Different authors then further extended this concept to semigroups  \cite{dem02} and nearrings \cite{cann05}. The zero-divisor graphs of semirings have been studied in \cite{atani08, dolzan, dolzan212}.

For an ordered subset $W=\{ w_1,w_2,\ldots,w_k\}$ of the vertex set of graph $G$ and a vertex $v$ of $G$, the $k$-vector
$r(v|W)=(d(v,w_1),d(v,w_2),\ldots,d(v,w_k))$ is called the \DEF{representation} of $v$ with respect to $W$. A set $W$ is called a \DEF{resolving set} for $G$ if distinct vertices of $G$ have distinct representations with respect to $W$. A resolving set of minimal cardinality for $G$ is a \DEF{basis} of $G$ and the cardinality of the basis is called the \DEF{metric dimension} of $G$, denoted by $\MD(G)$ \cite{chartrand}. The concept of a metric dimension was introduced by Slater \cite{slater} while studying the problem of uniquely determining the location of an intruder in a network. The metric dimension was then studied by Harary and Melter \cite{harary} and it has appeared in various applications of graph theory, for example pharmaceutical chemistry \cite{cameron, chartrand}, robot navigation in space \cite{khuller1996}, combinatorial optimization \cite{sebo} and sonar and coast guard long range navigation \cite{slater, slater2002}.
It turns out that determining the metric dimension of a graph remains a NP-complete problem even in special cases like bounded-degree planar graphs \cite{diaz}, or split graphs, bipartite graphs and their complements and line graphs of bipartite graphs \cite{epstein}.

The metric dimensions of graphs corresponding to various algebraic structures have been studied in many different settings. The metric dimension of a zero-divisor graph of a commutative ring was studied in \cite{ou21, pirzada, pirzada1, redmond21}, a total graph of a finite commutative ring in \cite{dolzan1}, an annihilating-ideal grah of a finite ring in \cite{dolzan21}, a commuting graph of a dihedral group in \cite{faisal}, etc.

In this paper, we study the metric dimension of the zero-divisor of a commutative entire antinegative semiring. A \emph{semiring} is a set $S$ equipped with binary operations $+$ and $\cdot$ such that $(S,+)$ is a commutative monoid 
with identity element 0 and $(S,\cdot)$ is a monoid with identity element 1. In addition, operations $+$ and $\cdot$
are connected by distributivity and 0 annihilates $S$. A semiring is \emph{commutative} if $ab=ba$ for all $a,b \in S$.

The theory of semirings has many applications in optimization theory, automatic control, models of discrete event networks and graph theory (see e.g. \cite{baccelli, cuninghame, li2014, zhao2010}).  For an extensive theory of semirings, we refer the reader to \cite{hebisch}. There are many natural examples of commutative semirings, for example, the set of nonnegative integers (or reals) with the usual operations of addition and multiplication. Other examples include distributive lattices, tropical semirings, dio\"{\i}ds, fuzzy algebras, inclines and bottleneck algebras. 
A semiring $S$ is called \emph{entire} if $ab=0$ for some $a, b \in S$ implies $a=0$ or $b=0$ and it is called \emph{antinegative} or \emph{zero-sum-free}, if $a+b=0$ for some $a, b \in S$ implies that $a=b=0$. Antinegative semirings are also called \emph{antirings}. The simplest example of an antinegative semiring  is the binary Boolean semiring $\BB$, the set $\{0,1\}$ in which addition and multiplication are the same as in $\ZZ$ except that $1+1=1$.

This paper is organized as follows. In the next section, we examine some preliminary graph-theoretical notions and results and state a preparatory lemma. In Section 3, we calculate the metric dimension of the zero-divisor graph of $M_n(\BB)$ (see Theorem \ref{metricboolean}). In the last section, we use this result to calculate the metric dimension of the zero-divisor graoh of $M_n(S)$, where $S$ is an arbitrary commutative entire antinegative semiring (see Theorem \ref{boolean}). 

\bigskip


\bigskip
\section{Preliminaries}
\bigskip

For a semiring $S$, we denote by $Z(S)$ the set of zero-divisors in $S$, $Z(S)=\{x \in S;$ there exists
$0 \neq y \in S \text { such that } xy=0 \text { or } yx=0 \}$. 
We denote by $\Gamma(S)$ the \emph{zero-divisor graph} of $S$. 
The vertex set $V(\Gamma(S))$ of $\Gamma(S)$ is the set of elements in $Z(S) \setminus \{0\}$ and 
an unordered pair of vertices $x,y \in V(\Gamma(S))$, $x \neq y$, is an edge 
$x-y$ in $\Gamma(S)$ if $xy = 0$ or $yx=0$.
The sequence of edges $x_0 - x_1$, $ x_1 - x_2$, ..., $x_{k-1} - x_{k}$ in a graph is called \emph{a path of length $k$}. The \DEF{distance} between vertices $x$ and $y$ is the length of the shortest
path between them, denoted by $d(x,y)$. 
The \DEF{diameter} $\diam{\Gamma}$ of the graph $\Gamma$ is the longest distance between any two vertices of the graph.

Next, we shall also need the following definition.

\bigskip

\begin{Def}
Let $v$ be a vertex of a graph $G$. Then the \emph{open neighbourhood} of $v$ is $N(v)=\{u \in V(G); \text{ there exists an edge } uv \text { in G}\}$ and the \emph{closed neighbourhood} of $v$ us $N[v]=N(v) \cup \{v\}$.  
Two distinct vertices $u$ and $v$ of $G$ are \emph{twins} if $N(u)=N(v)$ or $N[u]=N[v]$.
\end{Def}

\bigskip

The following lemma can be found in \cite{hernando}.

\bigskip

\begin{Lemma}[\cite{hernando}]\label{twins}
Suppose $u$ and $v$ are twins in a connected graph $G$ nad the set $W$ is a resolving set for $G$. Then $u$ or $v$ is in $W$.
\end{Lemma}

\bigskip

In this paper, we mostly concern ourselves with the semirings of matrices. For a semiring $S$, we denote by $M_n(S)$ the semiring of all $n$ by $n$ matrices with entries in $S$. We shall denote by $E_{ij} \in M_n(S)$ the matrix with $1$ at entry $(i,j)$ and zeros elsewhere.  
For a matrix $A \in M_n(S)$, we shall denote by $A_{ij} \in S$ the $(i,j)$-th entry of $A$.
Furthermore, let $N_n$ denote the
set $\{1,2,\ldots,n\}$. 

\bigskip

We shall see that the semiring of Boolean matrices $M_n(\BB)$ plays a special role in the search for the metric dimension of the zero-divisor graph of an entire antiring. 
Thus, the next section studies the metric dimension of  $\Gamma(M_n(\BB))$.

\bigskip
\bigskip

\section{The metric dimension of the zero-divisor graph of Boolean matrices}
\bigskip

In this section we study the metric dimension of  $\Gamma(M_n(\BB))$. Let's start with the simplest case of two by two matrices.

\begin{Example}
\label{twobytwo}
Let $W= \left\{ \left[ \begin{matrix}1 & 0 \\
1 & 0 \end{matrix}\right], \left[ \begin{matrix}1 & 1 \\
0 & 0 \end{matrix}\right]   \right\} \subset M_2(\BB)$. It can be easily verified that $W$ is a resolving set for $\Gamma(M_2(\BB))$, so $\MD(\Gamma(M_2(\BB))) \leq 2$. Since $\Gamma(M_2(\BB))$ has $8$ vertices and Theorem 1 from \cite{dolzan} states that $\diam{\Gamma(M_2(\BB))} \leq 3$, we conclude that no set of cardinality $1$ can be a resolving set for $\Gamma(M_2(\BB))$, therefore $\MD(\Gamma(M_2(\BB))) = 2$.
\end{Example}

\bigskip

Now, we turn on to the case $n \geq 3$. For every $I, J  \subseteq N_n$ we define the set of all matrices with their zero rows and columns prescribed by $I$ and $J$ respectively by
$T_{I,J}=\{A \in M_n(\BB); A_{ik}=A_{kj}=0 \text{ for every } k \in N_n, \text { for every } s \notin I \text { there exists } t \in N_n \text { such that } A_{st} \neq 0 \text { and for every }  u \notin J \text { there exists } v \in N_n \text { such that } A_{vu} \neq 0 \}$. Also, we denote by $t_{I,J}= |T_{I,J}|$ the cardinality of set $T_{I,J}$. 

The next two lemmas investigate the values of $t_{I,J}$ for various sets $I, J \subseteq N_n$.

\bigskip

\begin{Lemma}
\label{tgredonmanj2}
If $|I|=n-1$ or $|J|=n-1$ then $t_{I, J}=1$.
\end{Lemma}
\begin{proof}
Suppose $|I|=n-1$. Then any matrix $A \in T_{I,J}$ has only one non-zero row. In this row, the zero elements are exactly determined by the set $J$, so $T_{I,J}$ has exactly one element. Similary, we reason in the case $|J|=n-1$.
\end{proof}

\bigskip

\begin{Lemma}
\label{numberoft}
For every $I, J  \subsetneq N_n$, we have $t_{I,J}=\sum\limits_{k=0}^{n-|J|}{(-1)^k{n-|J| \choose k}\left(2^{n-|J|-k}-1 \right)^{n-|I|}}$.
\end{Lemma}
\begin{proof}
Notice that $t_{I,J}$ is exactly equal to the number of matrices of size $(n-|I|) \times (n - |J|)$ with entries in $\BB$ that have no zero row or column. 
So, let us examine the set of $(n-|I|) \times (n - |J|)$ matrices with elements in $\BB$. Obviously, there are exactly $(2^{n-|J|}-1)^{n-|I|}$ matrices that have all rows nonzero, but naturally some of them may have a few zero columns. Suppose therefore that a matrix has at least $k$ zero columns for some $k \in N_{n-|J|}$. We have $n-|J| \choose k$ possible ways to choose these $k$ columns. But if we disregard the zero columns, there are $2^{n-|J|-k}-1$ possible ways to choose the remaining elements in every (nonzero) row.  Since there are $n-|I|$ rows, this yields $(2^{n-|J|-k}-1)^{n-|I|}$ matrices. Now, if we sum this over all possible $k$, we will have counted some matrices (with more than $k$ zero columns) multiple times, but the inclusion exclusion principle then yields that there are exactly
$\sum_{k=0}^{n-|J|}{{n-|J| \choose k}(-1)^k(2^{n-|J|-k}-1)^{n-|I|}}$ matrices of size $(n-|I|) \times (n - |J|)$ with entries in $\BB$ that have no zero rows or columns.
\end{proof}

\bigskip

\begin{Remark}
\begin{enumerate}
\item
There is no known closed formula for the expression of Lemma \ref{numberoft} in the general case, even in the special case of square matrices (see \cite{seq}).
\item
Obviously, the number $t_{I,J}$ is only dependant on the cardinalites of sets $I$ and $J$. Therefore for any $1 \leq i,j \leq n-1$, we can define 
$t_{i, j}=t_{\{1,\ldots,i\},\{1,\ldots,j\}}$.
\end{enumerate}
\end{Remark}

\bigskip

The next lemma shows that all elements in $T_{I,J}$ are twins.

\bigskip

\begin{Lemma}
\label{twinsresolving}
For every $I, J  \subsetneq N_n$, all elements of the set $T_{I,J}$ are twins in $\Gamma(M_n(\BB))$.
\end{Lemma}
\begin{proof}
Choose $A \in T_{I,J}$ and observe that $AB=0$ for $B \in M_n(\BB)$ if and only if $B \in T_{K, L}$ for some $K, L \subseteq N_n$ with $N_n \setminus J \subseteq K$, and $CA=0$ for $C \in M_n(\BB)$ if and only if $C \in T_{K,L}$ for some $K, L \subseteq N_n$ with $N_n \setminus I \subseteq L$. Thus, all matrices in $T_{I,J}$ have the same neighbours in $\Gamma(M_n(\BB))$. 
\end{proof}

\bigskip


\bigskip

The next task is to give a construction of a resolving set for $\Gamma(M_n(\BB))$. We proceed thusly. 
For every $\emptyset \subsetneq I,J \subsetneq N_n$, let $X_{I,J}$ denote an arbitrarily chosen element of $T_{I,J}$. Also,   for every $\emptyset \subsetneq I \subsetneq N_n$ with $|I| \leq n-2$, let $Y_{I, \emptyset}$ denote  an arbitrarily chosen element of $T_{I, \emptyset}$ and  let $Z_{\emptyset, I}$ denote  an arbitrarily chosen element of $T_{\emptyset,I}$ .
Now, define $R$ with
$R= \left(\bigcup\limits_{\emptyset \subsetneq I,J \subsetneq N_n}{X_{I,J}}\right) \cup \left(\bigcup\limits_{\emptyset \subsetneq I \subsetneq N_n \& |I| \leq n-2}{Y_{I, \emptyset} \cup Z_{\emptyset, I}}\right) \cup T_{N_{n-1},\emptyset} \cup T_{\emptyset, N_{n-1}}$. 
Finally, define 
$W_R = \left( \bigcup\limits_{\emptyset \subseteq I,J \subsetneq N_n \& I \cup J \neq \emptyset}{T_{I,J}} \right) \setminus R$.
We shall prove that $W_R$ is a resolving set for $\Gamma(M_n(\BB))$ for every choice of the elements in set $R$. Before we do that, let us examine its cardinality.

\bigskip

\begin{Lemma}
\label{mocwr}
 $|W_R|=2(n-1)+\sum\limits_{i,j=0 \& ij \neq 0}^{n-2}{n \choose i}{n \choose j} \left[\sum\limits_{k=0}^{n-j}{(-1)^k{n-j \choose k}\left(2^{n-j-k}-1 \right)^{n-i}}-1\right]$.
\end{Lemma}
\begin{proof}
Observe firstly that by the construction of $R$, for any set $I \subseteq N_n$ with $|I|=n-1$, we have $T_{\emptyset,I} \subset R$ if and only if $T_{I, \emptyset} \subset R$ if and only if $I=N_{n-1}$. This implies that there are $2(n-1)$ sets of the form  $T_{\emptyset,I}$ and $T_{\emptyset,I}$ with $|I|=n-1$ that are subsets of $W_R$. Note that in this case  $|T_{\emptyset,I}|=|T_{\emptyset,I}|=1$ by  Lemma \ref{tgredonmanj2}, hence we have the first $2(n-1)$ elements in $W_R$.
Notice next that $R$ contains exactly one element from $T_{I,J}$ for every $\emptyset \subseteq I,J \subsetneq N_n$ with $I \cup J \neq \emptyset$ and $|I|, |J| \leq n-2$ and that there are exactly ${n \choose i}{n \choose j}$ different possible pairs of sets $I$ and $J$ in $N_n$ with $|I|=i$ and $|J|=j$. Lemma \ref{numberoft} now concludes this proof.
\end{proof}

\bigskip

\begin{Remark}
\label{kolikojedvojckov}
Note that by Lemmas \ref{twins} and \ref{twinsresolving} every resolving set for $\Gamma(M_n(\BB))$ has to contain all but perhaps one element from every set $T_{I,J}$.
This, by a similar calculation as in the proof of Lemma \ref{mocwr} yields $\sum\limits_{i,j=0 \& ij \neq 0}^{n-2}{n \choose i}{n \choose j} \left[\sum\limits_{k=0}^{n-j}{(-1)^k{n-j \choose k}\left(2^{n-j-k}-1 \right)^{n-i}}-1\right] = |W_R| - 2(n-1)$ distinct elements that belong to each and every resolving set of $\Gamma(M_n(\BB))$.
\end{Remark}

\bigskip

In order to find the metric dimension of $\Gamma(M_n(\BB))$ we shall also need the following two lemmas.

\bigskip

\begin{Lemma}
\label{distancetwo}
Suppose that $A, B \in M_n(\BB)$ are zero-divisors such that $A \in T_{I_A,J_A}$ for some sets $I_A,J_A \in N_n$ with $|I_A|, |J_A|=1$.
Then $d(A,B) \leq 2$.
\end{Lemma}
\begin{proof}
Denote $I_A=\{i\}$ and $J_A=\{j\}$. Since $B$ is a zero-divisor, $B \in T_{I_B,J_B}$ for some $I_B, J_B \subseteq N_n$ and there exists $k \in N_n$ such that $k \in I_B$ or $k \in J_B$. If $k \in I_B$ then $E_{jk}B=AE_{jk}=0$, so $d(A,B) \leq 2$. Similarly, if $k \in J_B$ then $BE_{ki}=E_{ki}A=0$ and again $d(A,B) \leq 2$.
\end{proof}

\begin{Lemma}
\label{distancethree}
Suppose that $A, B \in M_n(\BB)$ are zero-divisors such that either
\begin{enumerate}
\item
 $A \in T_{I_A,\emptyset}$ and $B \in T_{I_B,\emptyset}$ for some sets $I_A, I_B \subsetneq N_n$ with $I_A \cap I_B = \emptyset$, or
\item
 $A \in T_{\emptyset, J_A}$ and $B \in T_{\emptyset, J_B}$ for some sets $J_A, J_B \subsetneq N_n$ with $J_A \cap J_B = \emptyset$.
\end{enumerate}
Then $d(A,B) = 3$.
\end{Lemma}
\begin{proof}
Obviously it suffices to prove one of the two assertions, the other one can be done with a simetrical proof. Suppose therefore that $A \in T_{I_A,\emptyset}$ and $B \in T_{I_B,\emptyset}$ for some sets $I_A, I_B \subsetneq N_n$ with $I_A \cap I_B = \emptyset$. By \cite[Theorem 1]{dolzan}, $\diam{\Gamma(M_n(\BB))} \leq 3$. Since $AB \neq 0$ and $BA \neq 0$,  we only have to prove that $d(A,B) \neq 2$. Note that $AC, BC \neq 0$ for every $0 \neq C \in M_n(\BB)$. So, suppose there exists $C \in Z(M_n(\BB))$ such that $CA=CB=0$. But $CA=0$ implies $C \in T_{I_C, J_C}$ for some $I_C, J_C \subseteq N_n$ with $N_n \setminus I_A \subseteq J_C$. Similarly, $CB=0$ implies $N_n \setminus I_B \subseteq J_C$. But $I_A \cap I_B = \emptyset$, resulting in $J_C = N_n$, so $C=0$, a contradiction.
\end{proof}

\bigskip

We can now prove that $W_{R}$ is a resolving set for $\Gamma(M_n(\BB))$.

\bigskip

\begin{Theorem}
\label{resolvingboolean}
For any $n \geq 3$, the set $W_R$ is a resolving set for $\Gamma(M_n(\BB))$.
\end{Theorem}
\begin{proof}
Observe firstly that $Z(M_n(\BB)) \setminus \{0\} =  \bigcup\limits_{\emptyset \subseteq I,J \subsetneq N_n \& I \cup J \neq \emptyset}{T_{I,J}} $, so $Z(M_n(\BB)) \setminus \{0\} = W_R \cup R$.  Now, choose arbitrary $A, B \in R$. We have to prove that $A$ and $B$ have different representations with respect to $W_R$.
Denote by $I_A, J_A$ and $I_B, J_B$ the sets such that $A \in T_{I_A,J_A}$ and $B \in T_{I_B,J_B}$. By the definition of $R$, we have $I_A \neq I_B$ or $J_A \neq J_B$. Suppose without loss of generality that $I_A \neq I_B$. Now, we have $|I_A| \geq |I_B|$ or $|I_B| \geq |I_A|$, and again without loss of generality we can assume that $|I_A| \geq |I_B|$. 

We examine firstly the case that $|I_A| \geq 2$. We denote $J_C=N_n \setminus I_A \neq \emptyset$ and observe that 
$|J_C | \leq n-2$, thus $|T_{\emptyset,J_C}| \geq 2$. Therefore,
by the definition of $W_R$, there exists a matrix $C \in W_R$ such that $C \in T_{\emptyset,J_C}$. This implies that $CA=0$ and since $I_A \nsubseteq I_B$, we see that $CB \neq 0$. By definition, $C$ is not a right zero-divisor, so $d(A,C)=1$ and $d(B,C) \geq 2$ and therefore $A$ and $B$ have different representations with respect to $W_R$. 

Now, assume that $|I_A| \leq 1$. 
Since $I_A \neq I_B$, $I_A \neq \emptyset$, so it is left for us to examine the case $|I_A| = 1$. We have two possibilites, either $|I_B|=1$ or $I_B = \emptyset$. Assume first that $|I_B|=1$. Then $I_A \neq \{n\}$ or $I_B \neq \{n\}$. We can assume without loss of generality that $I_A \neq \{n\}$. 
Let us  denote $J_C=N_n \setminus I_A$ and observe that $|J_C|=n-1$ and $J_C \neq N_{n-1}$, so by the definition of $W_R$, there exists a matrix $C \in W_R$ such that $C \in T_{\emptyset,J_C}$. This implies that $CA=0$ and $CB \neq 0$. Since $C$ is not a right zero-divisor, we conclude that $d(A,C)=1$ and $d(B,C) \geq 2$, so $A$ and $B$ have different representations with respect to $W_R$. 
Finally, assume that $I_B = \emptyset$. Note that $J_B \neq N_n$, so choose any $k \notin J_B$ and observe that by the definition of $W_R$ and the fact that $n \geq 3$ there exists $C \in W_R$ such that $C \in T_{\emptyset, \{k\}}$. Now, Lemma \ref{distancethree} yields $d(B,C)=3$. Denote by $l$ the only element of $I_A$ and observe that $E_{kl}A=CE_{kl}=0$, so $d(A,C) \leq 2$, again concluding that $A$ and $B$ have different representations with respect to $W_R$. 
\end{proof}

\bigskip

We can now prove the following Theorem, which is the main result of this section.

\bigskip

\begin{Theorem}
\label{metricboolean}
For any $n \geq 2$, the metric dimension of $\Gamma(M_n(\BB))$ equals 
$$\MD(\Gamma(M_n(\BB)) =2(n-1)+\sum\limits_{i,j=0 \& ij \neq 0}^{n-2}{n \choose i}{n \choose j} \left[\sum\limits_{k=0}^{n-j}{(-1)^k{n-j \choose k}\left(2^{n-j-k}-1 \right)^{n-i}}-1\right].$$
\end{Theorem}
\begin{proof}
If $n=2$, Example \ref{twobytwo} tells us that $\MD(\Gamma(M_2(\BB)) =2$, so the result holds in this case. So, let us from now on assume that $n \geq 3$.
By Lemma \ref{mocwr}, we have $|W_R|=2(n-1)+\sum\limits_{i,j=0 \& ij \neq 0}^{n-2}{n \choose i}{n \choose j} \left[\sum\limits_{k=0}^{n-j}{(-1)^k{n-j \choose k}\left(2^{n-j-k}-1 \right)^{n-i}}-1\right]$ and by Theorem \ref{resolvingboolean}, we have $\MD(\Gamma(M_n(\BB)) \leq |W_R|$. 

Now, we have to prove that we also have $\MD(\Gamma(M_n(\BB)) \geq |W_R|$. So, let $W$ be an arbitrary resolving set for $\Gamma(M_n(\BB))$.
Remark \ref{kolikojedvojckov} assures us that there are at least $\sum\limits_{i,j=0 \& ij \neq 0}^{n-2}{n \choose i}{n \choose j} \left[\sum\limits_{k=0}^{n-j}{(-1)^k{n-j \choose k}\left(2^{n-j-k}-1 \right)^{n-i}}-1\right]$ "twin" elements in $W$ (all but perhaps one element from $T_{I,J}$ for every $0 \subseteq I,J \subsetneq N_n$ with $I \cup J \neq \emptyset$ and $|I|,|J| \leq n-2$).
Now, choose $i, k \neq j  \in N_n$ and define $I_A=\{i\}, J_A=\{j\}, I_B=\{i\}, J_B=\{k\}$. Let us examine two arbitrary matrices $A \in T_{I_A,J_A}, B \in T_{I_B,J_B}$. If neither of $A$ and $B$ is in $W$, then Lemma \ref{distancetwo} implies that the only way for $A$ and $B$ to have different representations according to $W$ is that there exists $C \in W$ such that $C$ is a neighbour of $A$, but not $B$, or vice versa. But such $C$ has to be either an element of 
$T_{N_n\setminus\{j\},J_C}$ or $T_{N_n\setminus\{k\},J_C}$ for some $N_n\setminus\{i\} \neq J_C \subsetneq N_n$. So, we now have one of the following four possibilites: every element of $T_{I_A,J_A}$ is in $W$, every element of $T_{I_B,J_B}$ is in $W$, some element of $T_{N_n\setminus \{j\},J_C}$ is in $W$, or some element of $T_{N_n\setminus \{k\},J_C}$ is in $W$. In every case, we see that $W$ has to contain one additional element to the "twin" elements described above. We can reason similarly in case $I_A=\{i\}, J_A=\{k\}, I_B=\{j\}, J_B=\{k\}$ for any $i \neq j, k \in N_n$, yielding again one of the four options: every element of $T_{I_A,J_A}$ is in $W$, every element of $T_{I_B,J_B}$ is in $W$, some element of $T_{I_C,N_n\setminus \{i\}}$ is in $W$, or some element of $T_{I_C,N_n\setminus \{j\}}$ is in $W$ (for some $N_n \setminus \{k\} \neq I_C \subsetneq N_n$). We can easily see that in this way, we must have at least $2(n-1)$ (distinct) additional elements in $W$ (one for any two distinct rows and one for any two distinct columns), thus
$|W| \geq 2(n-1)+\sum\limits_{i,j=0 \& ij \neq 0}^{n-2}{n \choose i}{n \choose j} \left[\sum\limits_{k=0}^{n-j}{(-1)^k{n-j \choose k}\left(2^{n-j-k}-1 \right)^{n-i}}-1\right]$, which proves the theorem. 
\end{proof}

\bigskip

\bigskip
\bigskip

\section{The general case}
\bigskip

Now, we examine the metric dimension of the  zero-divisor graph of $M_n(S)$ for an arbitrary commutative entire antinegative semiring $S$. We shall need the following definition.

\bigskip

\begin{Def}
Let $S$ be an entire antinegative semiring and $A \in M_n(S)$. Then the \emph{pattern} of $A$ is the matrix
$\pat(A) \in M_n(S)$ such that  for every $i, j \in N_n$ we have $\pat(A)_{ij}=1$ if and only if $A_{ij} \neq 0$ and $\pat(A)_{ij}=0$ otherwise.
\end{Def}

\bigskip

\begin{Lemma}
\label{pattwins}
Let $S$ be an entire antiring. For every $0 \neq A \in Z(M_n(S))$ matrices $A$ and $\pat(A)$ are twins in $\Gamma(M_n(S))$.
\end{Lemma}
\begin{proof}
This follows directly from the fact that $S$ is an entire antinegative semiring.
\end{proof}

\bigskip

\begin{Lemma}
\label{booleantwins}
Let $S$ be an entire antiring, $\Lambda \subseteq N_n \times N_n$ and $\alpha_{ij}, \beta_{ij} \in S \setminus \{0\}$ for all $(i,j) \in \Lambda$. If matrices $\sum_{(i,j) \in \Lambda}{\alpha_{ij}E_{ij}}$ and $\sum_{(i,j) \in \Lambda}{\beta_{ij}E_{ij}}$ are non-zero zero-divisors, they are twins in $\Gamma(M_n(S))$.
\end{Lemma}
\begin{proof}
Denote $A=\sum_{(i,j) \in \Lambda}{\alpha_{ij}E_{ij}}$ and $B=\sum_{(i,j) \in \Lambda}{\beta_{ij}E_{ij}}$. By Lemma \ref{pattwins}, both $A$ and  $\pat(A)$, as well as  $B$ and  $\pat(B)$ are twins in $\Gamma(M_n(S))$. By definition though,  $\pat(A)=\pat(B)$, so $A$ and $B$ are also twins.
\end{proof}

\bigskip

Lemmas \ref{twins} and \ref{booleantwins} imply that for infinite entire antirings $S$, the metric dimension of $\Gamma(M_n(S))$ is infinite. Therefore, we shall limit ourselves to studying finite semirings. Note that any entire antiring with two elements is isomorphic to $\BB$, so we can also assume that $|S| \geq 3$.

\bigskip

\begin{Theorem}
\label{boolean}
Let $S$ be an entire finite antiring with $|S| \geq 3$ and $n \geq 2$. Then the following formula holds:
 $\MD(\Gamma(M_n(S)))=|S|^{n^2}-2^{n^2}-\sum\limits_{k=0}^{n}{(-1)^k{n \choose k}\left[\left(|S|^{n-k}-1 \right)^{n}-\left(2^{n-k}-1 \right)^{n}\right]}+\MD(\Gamma(M_n(\BB)))- 2(n-1).$
\end{Theorem}
\begin{proof}
Throughout this proof we shall (by a slight abuse of notation) consider that $\BB$ is a proper subsemiring in $S$.
Suppose $W$ is a resolving set for $\Gamma(M_n(S))$. Lemmas \ref{twins} and \ref{booleantwins} imply that for every $\Lambda \subseteq N_n \times N_n$ at most one zero-divisor matrix with its pattern of zero and non-zero entries prescribed by $\Lambda$ is not in $W$. This yields $| Z(M_n(S))| - | Z(M_n(\BB))|$ elements that have to be included in $W$. If $n \geq 3$, Remark \ref{kolikojedvojckov}, Theorem \ref{metricboolean} and Lemma \ref{pattwins} yield additional $\MD(\Gamma(M_n(\BB))) - 2(n-1)$ elements that have to be in $W$. If $n=2$, $\MD(\Gamma(M_2(\BB))) = 2(2-1)$, so in both cases $\MD(\Gamma(M_n(S))) \geq | Z(M_n(S))| - | Z(M_n(\BB))|+\MD(\Gamma(M_n(\BB))) - 2(n-1)$.

On the other hand, let $W = Z(M_n(S)) \setminus Z(M_n(\BB))$ and if $n \geq 3$, add to $W$ also  the $\MD(\Gamma(M_n(\BB)))- 2(n-1)$ elements from Remark \ref{kolikojedvojckov}. By construction, $|W| = | Z(M_n(S))| - | Z(M_n(\BB))| + \MD(\Gamma(M_n(\BB))) - 2(n-1)$ for every $n \geq 2$. Let us prove that $W$ is a resolving set. Choose arbitrary non-zero matrices $A \neq B \in Z(M_n(S))$. Now, if $A$ and $B$ have the same pattern, one of them is in $W$ by definition. Suppose therefore that $A$ and $B$ have different patterns. Remark \ref{kolikojedvojckov} now ensures that if matrices $\pat(A)$ and $\pat(B)$ belong to the same set $T_{I,J}$ for some $\emptyset \subseteq I,J \subsetneq N_n$, at least one of them is in $W$. So, let us assume that $\pat(A) \in T_{I_A,J_A}$ and $\pat(B) \in T_{I_B,J_B}$, where $I_A \neq I_B$ or $J_A \neq J_B$. Without loss of generailty (if necessary, swapping the roles of rows and columns, or matrices  $A$ and $B$ respectively), assume that $I_A \neq I_B$ and $|I_A| \leq |I_B|$. Since $|S| \geq 3$, at least one zero-divisor matrix with every pattern of its zero and non-zero entries is in $W$, so we can choose $C \in W$ such that $\pat(C) \in T_{\emptyset, N_n \setminus I_B}$. Therefore $\pat(C)$ is not a right zero-divisor, but $\pat(C)\pat(B)=0$, and since $I_B \nsubseteq I_A$ also $\pat(C)\pat(A) \neq 0$.
Lemma \ref{pattwins} now implies that $A$ and $B$ have different representations with respect to $W$.
This proves that $\MD(\Gamma(M_n(S))) \leq | Z(M_n(S))| - | Z(M_n(\BB))|+\MD(\Gamma(M_n(\BB)))-2(n-1)$.

Finally, observe that the zero divisors in $M_n(S)$ are exactly those matrices with at least one zero row or column, so with a similar argument as in the proof of Lemma \ref{numberoft}, we can prove that $| Z(M_n(S))|=|S|^{n^2}-\sum\limits_{k=0}^{n}{(-1)^k{n \choose k}\left(|S|^{n-k}-1 \right)^{n}}$. This yields that $| Z(M_n(S))|-| Z(M_n(\BB))|=|S|^{n^2}-2^{n^2}-\sum\limits_{k=0}^{n}{(-1)^k{n \choose k}\left[\left(|S|^{n-k}-1 \right)^{n}-\left(2^{n-k}-1 \right)^{n}\right]}$, which proves the assertion.
\end{proof}

\bigskip


\end{document}